\documentclass{amsart}
\usepackage{amsthm}
\usepackage{amsmath}
\usepackage{amssymb}
\usepackage{enumerate}
\usepackage{enumitem}
\usepackage{mathrsfs}
\usepackage{float}
\usepackage{graphicx}
\usepackage{mathtools}
\usepackage{hyperref}
\usepackage{caption}
\usepackage{subcaption}
\usepackage{faktor}
\usepackage[dvipsnames]{xcolor}
\usepackage{lineno}

\newtheorem{theorem}{Theorem}
\newtheorem{corollary}[theorem]{Corollary}
\newtheorem{proposition}[theorem]{Proposition}
\newtheorem{lemma}[theorem]{Lemma}
\newtheorem{definition}[theorem]{Definition}

\newtheorem{remark}[theorem]{Remark}

\numberwithin{theorem}{section}
\numberwithin{equation}{section}

\newcounter{relctr} 
\everydisplay\expandafter{\the\everydisplay\setcounter{relctr}{0}} 

\AtBeginDocument{} 

\makeatletter
\newcommand{\mylabel}[2]{#2\def\@currentlabel{#2}\label{#1}}
\makeatother

\allowdisplaybreaks

\begin{document}
	\title[Extrapolation on capacitary spaces]{An Extrapolation Theorem in the setting of Hausdorff capacities}
	\author[A Deshmukh]{Aniruddha Deshmukh}
	\address{Department of Mathematics, Indian Institute of Technology Indore, Khandwa Road, Simrol, Indore, Madhya Pradesh-453552, India.}
	\email{aniruddha480@gmail.com}
	\date{\today}
	\begin{abstract}
In this article, we prove an analogue of the Rubio de Francia's extrapolation theorem in the setting of Hausdorff capacities. We prove the result using techniques analogous to those in the classical setting and using the recently developed theory of capacitary Muckenhoupt weights.
\end{abstract}
\maketitle
	\section{Introduction}
		\label{IntroductionSection}
		Capacities are an important aspect in the modern analysis, especially in potential theory. The notion of a capacity gives us a way to measure the ``size" of sets that might not be Lebesgue measurable. Cpacities are inherently non-linear objects and therefore provide a natural setting for many results in non-linear potential theory (for details, we refer the reader to \cite{Adams}). Choquet was one among the first to introduce the notion of a ``capacitable" set (see \cite{Choquet}). By a ``capacity" we mean a monotone set function $C: \mathcal{P} \left( \mathbb{R}^n \right) \rightarrow \left[ 0, \infty \right]$ with $C \left( \emptyset \right) = 0$ and that satisfies countable subadditivity. That is, for a countable collection $\left( E_k \right)_{k \in \mathbb{N}}$ of subsets of $\mathbb{R}^n$, we have,
		$$C \left( \bigcup\limits_{k \in \mathbb{N}} E_k \right) \leq \sum\limits_{k \in \mathbb{N}} C \left( E_k \right).$$
		Choquet, in \cite{Choquet}, defined capacities with the conditions of inner and outer regularity. It was pointed out by the authors of \cite{PonceSpector} that the regularity assumptions are often unnatural and the capacities that arise in practice are often neither inner regular nor outer regular. We, therefore, do not dwell much upon these assumptions. With the definition of capacities at hand, one may define integrals of non-negative functions with respect to capacities. Such integrals are now called \textit{Choquet integrals}. We give their precise definition here.
		\begin{definition}[Choquet Integral]
			\label{ChoquetIntegralDefinition}
			Let $C: \mathcal{P}(\mathbb{R}^n) \to [0,\infty]$ be a monotone set function. Then, the Choquet integral of a function $f : \mathbb{R}^n \to [0,\infty]$ with respect to $C$ is given by
			\begin{equation}\label{Choquet}
				\int_{\mathbb{R}^n} f\,\mathrm{d}C := \int_{0}^{\infty} C(\{x \in \mathbb{R}^n : f(x) > t\})\,\mathrm{d}t,
			\end{equation}
			where the right-hand side is a Lebesgue integral.
		\end{definition}
		If the integral given in Equation \eqref{Choquet} is finite, we say that $f$ is \textbf{Choquet integrable with respect to $C$}. With the notion of Choquet integrability at our disposal, we define function spaces in the setting of capacities analogous to the Lebesgue spaces in the classical setting. The defintion of capacitary $L^p$ spaces can be found in \cite{SpectorMaximal}. For a non-negative, monotone set function $C$ and for $1\leq p < \infty$, we define 
		\begin{equation}
			L^p(\mathbb{R}^n,C) := \left\{f: \mathbb{R}^n \to [-\infty,\infty] \Big| f \text{ is quasi-continuous and }\int_{\mathbb{R}^n} |f|^p\,\mathrm{d}C < \infty \right\}.
		\end{equation} 
		Here $f$ represents an equivalence class of quasi-continuous functions, equal to $f, C$-quasi everywhere (q.e.) in $\mathbb{R}^n$. By $f=g$ (q.e.), we mean that  $f(x) = g(x)$ except on a set $E$ with $C(E) = 0$. We refer the reader to  Definition $3.1$ in \cite{PonceSpector} for the definition of quasi-continuous functions. It is well-known that the Choquet integrals are, in general, not even sublinear, let alone linear. Consequently, capacitary Lebesgue spaces $L^p(\mathbb{R}^n,C)$ might not even be vector spaces for general capacities. However, under extra assumption on $C$, such as strong subadditivity (see Theorem 1.1 of \cite{PonceSpector}), it can be shown that $L^p(\mathbb{R}^n,C)$ is a real Banach space. We refer the reader to \cite{PonceSpector} for more details. Moreover, under the assumption of strong subadditivity, the quantity 
		\begin{equation}\label{norm}
			\|f\|_{L^p(\mathbb{R}^n,C)} := \left( \int_{\mathbb{R}^n}^{}|f|^p\,\mathrm{d}C \right)^\frac{1}{p},
		\end{equation}
		is a norm on $L^p(\mathbb{R}^n,C)$. In this article we use the symbol  $\|\cdot\|_{L^p(\mathbb{R}^n,C)}$ to denote the right-hand-side of of Equation \eqref{norm} for general capacities as well. It is to be understood that, in this setup, the symbol might not always mean a norm. For a non-negative monotone set function $C$ one may also define spaces of locally integrable functions as follows
		\begin{equation}
			\begin{aligned}
				L_{loc}^1(C) : = \Bigg\{ f: \mathbb{R}^n \to [-\infty,\infty] \Big| &f \text{ is quasi-continuous and} \\
				&\text{for each compact }K \subseteq \mathbb{R}^n, \int_{K}^{}|f|\hspace{0.1cm}\mathrm{d}C < \infty \Bigg\}.
			\end{aligned}
		\end{equation}
		One of the important capacities in literature is the \textit{Hausdorff content}. It arises in the study of Hausdorff measure and Hausdorff dimension (see for instance, p. 132 of \cite{Adams}). The latter are important concepts in the study of geometric measure theory. We refer the reader to \cite{Mattila1} and \cite{Mattila2} for further details on these topics. The Hausdorff content can be defined in terms of open balls or in terms of cubes. We give their definitions here.
		\begin{definition}[Hausdorff content]
			\label{HausdorffContent}
			The $\beta$-dimensional Hausdorff content is a set function defined on the power set of $\mathbb{R}^n$ as 
			\begin{equation}\label{Hausdorff}
				H_\infty^\beta(E) := \inf\left\{\sum_{j=1}^\infty \omega_\beta r_j^\beta : E \subseteq \bigcup_{j=1}^\infty B_j \right\},
			\end{equation}
			where $B_j$ is an open ball with radius $r_j$, $\omega_\beta = \frac{\pi^{\beta/2}}{\Gamma(\beta/2 +1)}$ and $0< \beta \leq n$.
		\end{definition}
		\begin{definition}[Cubic Hasudorff Content]
			\label{CubicHC}
			The cubic Hausdorff content is a set function $H_\infty^{\beta,c}: \mathcal{P}(\mathbb{R}^n) \rightarrow [0,\infty],$ defined as
			\begin{equation}
				H_\infty^{\beta,c} (E) := \inf\left \{\sum_{j=1}^{\infty}\ell (Q_j)^\beta : E \subseteq \bigcup_{i=1}^\infty Q_j \right\},
			\end{equation}
			where, $Q_j$ is a cube in $\mathbb{R}^n$ with sides parallel to the coordinate axes and side length $ \ell(Q_j)$. 
		\end{definition}
		One can easily see (due to the topological equivalence of cubes and open balls) that the Hausdorff content $H^{\beta}_{\infty}$ and the cubic Hausdorff content $H^{\beta, c}_{\infty}$ are equivalent. In the sequel, we use the same symbol $H^{\beta}_{\infty}$ to mean either of these definitions, when we do not care about the constants involved in the expressions.
		
		Apart from these definitions, one may also define Hausdorff contents using dyadic cubes. The precise definitions are as follows.
		\begin{definition}[Dyadic Hausdorff Content]
			Let $\mathcal{D}(Q_0)$ be the dyadic lattice associated to a cube $Q_0 \subseteq \mathbb{R}^n.$ The dyadic Hausdorff content is a set function $H_\infty ^{\beta,Q_0} : \mathcal{P}(\mathbb{R}^n) \rightarrow [0,\infty],$ defined by
			\begin{equation}\label{dyadic}
				H_\infty^{\beta,Q_0} (E) : = \inf\left\{\sum_{i=1}^{\infty}\ell(Q_i)^\beta : E \subseteq \bigcup_{i=1}^{\infty}Q_i, Q_i \in \mathcal{D}(Q_0)\right\},
			\end{equation}
			where $\ell(Q_i)$ is the side length of dyadic cube $Q_i.$
		\end{definition}
		It was shown in \cite{PonceSpector} that the dyadic Hausdorff content $H_{\infty}^{\beta, Q_0}$ and the Hausdorff content $H^{\beta}_{\infty}$ are equivalent and hence can be used interchangeably when one deals with (norm) inequalities. It was, in fact, observed by the authors of \cite{PonceSpector} that the dyadic Hausdorff content is better behaved that the Hausdorff content, in that it is strongly subadditive. As a consequence, the Choquet integral with respect to the dyadic Hausdorff content is sublinear. This becomes important later in our study. Also, for the ease of its handling, we shall always consider $Q_0 = \left[ 0, 1 \right)^n$ to define the dyadic Hausdorff content.
		
		With the notion of ``size" of sets (given by capacities) and the Choquet integral at one's disposal, one may talk about averaging and maximal operators with respect to capacities. This problem was taken up by the authors of \cite{SpectorMaximal} and later was revisited by the authors of \cite{Basak}. Although the (Hardy-Littlewood) maximal operator may be defined for general capacities, we give the definition only in the context of Hausdorff capacities, which is the main topic of this article.
		\begin{definition}[Hardy-Littlewood Maximal Operator]
			For a locally integrable function $f: \mathbb{R}^n \rightarrow \mathbb{R}$ with respect to the $\beta$-dimensional Hausdorff content $H^{\beta}_{\infty}$, we define the Hardy-Littlewood maximal function of $f$ as
			\begin{equation}
				\label{HLMaxFn}
				M_{\beta}f \left( x \right) := \sup\limits_{x \in Q} \left\lbrace \frac{1}{H^{\beta}_{\infty} \left( Q \right)} \int\limits_{Q} |f| \ \mathrm{d}H^{\beta}_{\infty} \right\rbrace,
			\end{equation}
			where, the supremum is taken over all cubes $Q$ containing $x$.
		\end{definition}
		\begin{remark}
			\normalfont
			In the classical setting, definition of the type given in Equation \eqref{HLMaxFn} is often referred to as ``\textit{non-centred cubic maximal function}". The classical Hardy-Littlewood maximal function is defined by taking the supremum over balls centered at $x$ instead of cubes. Again, it is easy to observe that these definitions are equivalent and can be used interchangeably when considering the problem of norm inequalities.
		\end{remark}
		Chen et al. in \cite{SpectorMaximal} proved that the maximal operator $M_{\beta}$ is bounded on the capacitary spaces $L^p \left( \mathbb{R}^n, H_{\infty}^{\beta} \right)$ by making use of techniques analogous to those in the classical setting. As a next step, the study of Muckenhoupt weights in the setting of Hausdorff capacities was natural. While in this classical setting the definition of weighted measure follows easily from the Reisz-Markov-Kakutani representation theorem, defining weighted capacities seems to be more intricate and subtle. One of the ways to realize the weighted Hausdorff content is (see \cite{Tang}),
		\begin{equation}\label{hmm}
			H_\omega^\beta (E) = \inf \sum_{j=1}^{\infty}\frac{\omega(Q_j)}{|Q_j|}\ell(Q_j)^\beta,
		\end{equation}
		where the infimum is taken over all coverings of $E$ by countable families of cubes $Q_j,$ with sides parallel to the coordinate axes with length $\ell(Q_j)$ and $\omega(Q) = \int_{Q}^{}\omega(x)\;\mathrm{d}x.$ However, it was pointed out by the author of \cite{Saito} (see Remark 1.4 of their article) that obtaining weak or strong inequalities for the maximal operator with $H_\omega^{\beta}$ as given in Equation \eqref{hmm} is difficult since there is no comparison between integrals with respect to $\mathrm{d}H_\omega^\beta$ and $\omega \ \mathrm{d}H^\beta$.
		
		The observations above led the authors of \cite{Huang} to define a new type of weighted Hausdorff content as follows. Given a non-negative function $w$ on $\mathbb{R}^n$, one defines the weighted Hausdorff capacity as the set function
		\begin{equation}
			\label{WeightedCapacity}
			w_{H^{\beta}_{\infty}} \left( E \right) := \int\limits_{E} w \left( x \right) \mathrm{d}H^{\beta}_{\infty}.
		\end{equation}
		 The study of weighted boundedness of the maximal operator $M_{\beta}$ was conducted by Haung et al. in \cite{Huang}, where the authors characterized all the weight functions $w$ that make the Hardy-Littlewood maximal operator $M_{\beta}$ bounded on $L^p \left( w \ \mathrm{d}H^{\beta}_{\infty} \right)$. Using their definition of the $A_{p, \beta}$ condition, the authors were able to show that the integrals with respect to $\mathrm{d}w_{H^{\beta}_{\infty}}$ and $w \ \mathrm{d}H^{\beta}_{\infty}$ are, in fact, comparable. They were also able to show other important results concerning the capacitary Muckenhoupt weights, such as the self-improving property and the Jones factorization theorem. We state here the results from \cite{Huang} that are important for this article.
		\begin{theorem}[Muckenhoupt $A_{p, \beta}$ weights]
			\label{ApWeightsTheorem}
			Let $0 < \beta \leq n$ and $1 < p < \infty$. Let $w$ be a weight function. Then, the following statements are equivalent.
			\begin{enumerate}
				\item There is a constant $K > 0$ such that for every $f \in L^p \left( \mathbb{R}^n, w \ \mathrm{d}H^{\beta}_{\infty} \right)$, we have,
				\begin{equation}
					\label{StrongPPInequality}
					\int\limits_{\mathbb{R}^n} \left| M_{\beta}f \left( x \right) \right|^p w \left( x \right) \mathrm{d}H^{\beta}_{\infty} \leq K \int\limits_{\mathbb{R}^n} \left| f \left( x \right) \right|^p w \left( x \right) \mathrm{d}H^{\beta}_{\infty}.
				\end{equation}
				\item There is a positive constant $A > 0$ such that for every cube $Q \subseteq \mathbb{R}^n$, we have,
				\begin{equation}
					\label{ApCondition}
					\left( \frac{1}{H^{\beta}_{\infty} \left( Q \right)} \int\limits_{Q} w \ \mathrm{d}H^{\beta}_{\infty} \right) \left( \frac{1}{H^{\beta}_{\infty} \left( Q \right)} \int\limits_{Q} w^{- \frac{1}{p - 1}} \mathrm{d}H^{\beta}_{\infty} \right)^{p - 1} \leq A.
				\end{equation}
			\end{enumerate}
		\end{theorem}
		\begin{remark}
			\normalfont
			Equation \eqref{ApCondition} is called the ``\textit{$A_{p, \beta}$ condition}".
		\end{remark}
		Similar to the classical setting, the authors of \cite{Huang} also define the $A_{1, \beta}$ condition as follows.
		\begin{definition}[$A_1$ condition]
			\label{A1ConditionDefinition}
			A weight function $w$ is said to satisfy the $A_1$ condition if there is a constant $K > 0$ such that for $H^{\beta}_{\infty}$-quasi-every $x \in \mathbb{R}^n$, we have
			\begin{equation}
				\label{A1Condition}
				M_{\beta}w \left( x \right) \leq K w \left( x \right).
			\end{equation}
		\end{definition}
		Using the $A_{1, \beta}$ condition, we get the weighted weak type $(1, 1)$ bounds for the maximal operator $M_{\beta}$.
		\begin{theorem}[\cite{Huang}]
			\label{WekTypeBoundedness}
			There exists a constant $K > 0$ such that for every $f \in L^1 \left( \mathbb{R}^n, w \ \mathrm{d}H^{\beta}_{\infty} \right)$, we have,
			\begin{equation}
				\label{Weak11}
				w_{H^{\beta}_{\infty}} \left( \left\lbrace x \in \mathbb{R}^n | M_{\beta}f \left( x \right) > t \right\rbrace \right) \leq \frac{K}{t} \| f \|_{L^1 \left( \mathbb{R}^n, w \ \mathrm{d}H^{\beta}_{\infty} \right)},
			\end{equation}
			if and only if $w$ satisfies the $A_{1, \beta}$ condition.
		\end{theorem}
		With the introduction of Muckenhoupt weights in the setting of Hausdorff capacities, we ask the question about extrapolation of boundedness from one $L^p$ space to all Lebesgue spaces. The classical result is due to Rubio de Francia (see \cite{RDF}). In this article, we give its analogue in the setting of Hausdorff capacities. In doing so, we prove an elementary results about construction of $A_{1, \beta}$ weights that is not included in \cite{Huang}.
		
		The article is organised as follows: in Section \ref{PreliminariesSection}, we give some prepratory results and recall certain properties about the Hausdorff content and the $A_{p, \beta}$ weights. In Section \ref{MainResultSection}, we first give a construction of the $A_{1, \beta}$ weights using the Hardy-Littlewood maximal function $M_{\beta}$. Then, we prove an extrapolation theorem for the weighted capacitary Lebesgue spaces.
	\section{Preliminaries}
		\label{PreliminariesSection}
		In this section, we discuss certain elementary results about Hausdorff content and $A_{p, \beta}$ weights. Before we begin with our results, we recall the H\"{o}lder inequality in the setting of capacities. This result can be found in \cite{Cerda} (see Theorem 2 of their article).
		\begin{theorem}[H\"{o}lder inequality]
			\label{HolderInequalityTheorem}
			Let $C$ be a strongly subadditive capacity on $\mathbb{R}^n$ (so that the Choquet integral is sublinear). Then, for functions $f, g: \mathbb{R}^n \rightarrow \mathbb{R}$ and $1 \leq p \leq \infty$, we have,
			\begin{equation}
				\label{HOlderInequality}
				\int\limits_{\mathbb{R}^n} \left| fg \right| \mathrm{d}C \leq \left( \int\limits_{\mathbb{R}^n} \left| f \right|^p \mathrm{d}C \right)^{\frac{1}{p}} \left( \int\limits_{\mathbb{R}^n} \left| g \right|^{p'} \mathrm{d}C \right)^{\frac{1}{p'}},
			\end{equation}
			where, $\frac{1}{p} + \frac{1}{p'} = 1$.
		\end{theorem}
		\begin{remark}
			\normalfont
			It is known that (see for instance, \cite{PonceSpector}) the dyadic Hausdorff capacity $H^{\beta, Q_0}_{\infty}$ is strongly subadditive. Therefore, the H\"{o}lder's inequality holds for $C = H^{\beta, Q_0}_{\infty}$.
		\end{remark}
		We now make the following elementary observation about the weighted $L^p$ norm of a non-negative function.
		\begin{lemma}
			\label{DualityLemma}
			Let $w: \mathbb{R}^n \rightarrow \left[ 0, \infty \right)$ be a weight function and $f: \mathbb{R}^n \rightarrow \left[ 0, \infty \right)$ be a function with $f \in L^p \left( \mathbb{R}^n, w \ \mathrm{d}H^{\beta, Q_0}_{\infty} \right)$. Then, we have,
			\begin{equation}
				\label{DualityNorm}
				\| f \|_{L^p \left( \mathbb{R}^n, w \ \mathrm{d}H^{\beta, Q_0}_{\infty} \right)} = \sup \left\lbrace \int\limits_{\mathbb{R}^n} fuw \ \mathrm{d}H^{\beta, Q_0}_{\infty} \right\rbrace,
			\end{equation}
			where, the supremum is taken over those $u \in L^{p'} \left( \mathbb{R}^n, w \ \mathrm{d}H^{\beta, Q_0}_{\infty} \right)$ with $u \geq 0$ and $\| u \|_{L^{p'} \left( \mathbb{R}^n, w \ \mathrm{d}H^{\beta, Q_0}_{\infty} \right)} = 1$.
		\end{lemma}
		\begin{proof}
			First, let us observe that for any $u \in L^{p'} \left( \mathbb{R}^n, w \ \mathrm{d}H^{\beta, Q_0}_{\infty} \right)$ with $u \geq 0$ and $\| u \|_{L^{p'} \left( \mathbb{R}^n, w \ \mathrm{d}H^{\beta, Q_0}_{\infty} \right)} = 1$, we have using Equation \eqref{HOlderInequality},
			$$\int\limits_{\mathbb{R}^n} fuw \ \mathrm{d}H^{\beta, Q_0}_{\infty} \leq \| f \|_{L^p \left( \mathbb{R}^n, w \ \mathrm{d}H^{\beta, Q_0}_{\infty} \right)}.$$
			This gives us
			$$\sup \left\lbrace \int\limits_{\mathbb{R}^n} fuw \ \mathrm{d}H^{\beta, Q_0}_{\infty} \right\rbrace \leq \| f \|_{L^p \left( \mathbb{R}^n, w \ \mathrm{d}H^{\beta, Q_0}_{\infty} \right)}.$$
			Here, the supremum is taken as mentioned in the statement of the result.
			
			On the other hand, let us observe that given $f \in L^p \left( \mathbb{R}^n, w \ \mathrm{d}H^{\beta, Q_0}_{\infty} \right)$, the function $u = \left( \frac{f}{\| f \|_{L^p \left( \mathbb{R}^n, w \ \mathrm{d}H^{\beta, Q_0}_{\infty} \right)}} \right)^{p - 1}$ is such that $u \geq 0$, $u \in L^{p'} \left( \mathbb{R}^n, w \ \mathrm{d}H^{\beta, Q_0}_{\infty} \right)$ and $\| u \|_{L^{p'} \left( \mathbb{R}^n, w \ \mathrm{d}H^{\beta, Q_0}_{\infty} \right)} = 1$. Further, it is easy to see that
			$$\int\limits_{\mathbb{R}^n} fuw \ \mathrm{d}H^{\beta, Q_0}_{\infty} = \| f \|_{L^p \left( \mathbb{R}^n, w \ \mathrm{d}H^{\beta, Q_0}_{\infty} \right)}.$$
			This completes the proof!
		\end{proof}
		\begin{remark}
			\normalfont
			While Lemma \ref{DualityLemma} mimics the duality of $L^p$ spaces in the measure theoretic setting, the proof we have presented only works for non-negative functions. Indeed, the theory of Choquet integrals for general (singed) functions is much more intricate and would require an independent study. Since we deal with maximal operators and norm estimates, the form of duality given for non-negative functions is sufficient to our needs.
		\end{remark}
		An important consequence of the boundedness of the maximal operator is that one gets a Lebesgue differentiation theorem in the setting of Hausdorff capacities. The same was proved in \cite{Basak} (see Theorem c in their article). We state it here for our use.
		\begin{theorem}[Lebesgue Differentiation Theorem]
			\label{LDT}
			For $f \in L^1 \left( \mathbb{R}^n, H^{\beta, Q_0}_{\infty} \right)$, we have,
			\begin{equation}
				\label{LDTEquation}
				\lim\limits_{Q \rightarrow x} \frac{1}{H^{\beta, Q_0}_{\infty} \left( Q \right)} \int\limits_{Q} \left| f \left( y \right) - f \left( x \right) \right| \mathrm{d}H^{\beta, Q_0}_{\infty} \left( y \right) = 0,
			\end{equation}
			where the limit is taken over a dyadic tower of cubes shrinking to $x$.
		\end{theorem}
		\begin{remark}
			\normalfont
			It is clear that the Lebesgue Differentiation Theorem holds for $f \in L^1_{\text{loc}} \left( H^{\beta, Q_0}_{\infty} \right)$, since we can replace $f$ by $f \cdot \chi_{\tilde{Q}}$ for some (large) dyadic cube $\tilde{Q}$ so that $f \cdot \chi_{\tilde{Q}} \in L^1 \left( \mathbb{R}^n, H^{\beta, Q_0}_{\infty} \right)$.
		\end{remark}
		The following consequence of the Lebesgue differentiation theorem is easy to observe.
		\begin{corollary}
			\label{LDTCor}
			There exists a constant $K > 0$ such that for any $f \in L^1_{\text{loc}} \left( H^{\beta}_{\infty} \right)$, we have for any $x \in \mathbb{R}^n$, $\left| f \left( x \right) \right| \leq K \ M_{\beta}f \left( x \right)$.
		\end{corollary}
		\begin{proof}
			First, we notice that for any dyadic cube $Q$ containing $x$,
			\begin{align*}
				\left| f \left( x \right) \right| &= \frac{1}{H^{\beta, Q_0}_{\infty} \left( Q \right)} \int\limits_{Q} \left| f \left( x \right) \right| \mathrm{d}H^{\beta, Q_0}_{\infty} \\
				&\leq \frac{1}{H^{\beta, Q_0}_{\infty} \left( Q \right)} \int\limits_{Q} \left| f \left( y \right) - f \left( x \right) \right| \mathrm{d}H^{\beta, Q_0}_{\infty} \left( y \right) + \frac{1}{H^{\beta, Q_0}_{\infty} \left( Q \right)} \int\limits_{Q} \left| f \left( y \right) \right| \mathrm{d}H^{\beta, Q_0}_{\infty} \left( y \right).
			\end{align*}
			Due to equivalence of $H^{\beta, Q_0}_{\infty}$ and $H^{\beta}_{\infty}$, we get
			$$\left| f \left( x \right) \right| \leq \frac{1}{H^{\beta, Q_0}_{\infty} \left( Q \right)} \int\limits_{Q} \left| f \left( y \right) - f \left( x \right) \right| \mathrm{d}H^{\beta, Q_0}_{\infty} \left( y \right) + K \ M_{\beta}f \left( x \right),$$
			for some constant $K > 0$ that depends only on $n$ and $\beta$. The result now follows from Equation \eqref{LDTEquation} by taking limits $Q \rightarrow x$ above.
		\end{proof}
		We also require the following property of $A_{p, \beta}$ weights in the sequel.
		\begin{proposition}
			\label{ApApPrimeCharacterization}
			For $p > 1$, a weight function $w \in A_{p, \beta}$ if and only if $w^{1 - p'} \in A_{p', \beta}$.
		\end{proposition}
		\begin{proof}
			For $p > 1$, we notice that $\omega^{1-p'} \in A_{p', \beta}$ is equivalent to
			\begin{equation*}
				\left( \frac{1}{H^{\beta}_{\infty} \left( Q \right)} \int\limits_{Q} w^{1 - p'} \mathrm{d}H^{\beta}_{\infty} \right)
				\left( \frac{1}{H_\infty^{\beta}(Q)} \int_Q \omega^{(1-p')(1-p)} \,\mathrm{d}H_\infty^{\beta} \right)^{p'-1} \leq A,
			\end{equation*}
			for every cube $Q \subseteq \mathbb{R}^n$. 
			By taking power $(p-1)$ and simplifying, we get,
			\begin{equation*}
				\left( \frac{1}{H^{\beta}_{\infty} \left( Q \right)} \int\limits_{Q} w^{1 - p'} \mathrm{d}H^{\beta}_{\infty} \right)^{p-1}
				\left( \frac{1}{H_\infty^{\beta}(Q)} \int_Q \omega \,\mathrm{d}H_\infty^{\beta} \right) \leq A^{p - 1}.
			\end{equation*}
			However, this equivalent to $\omega\in A_{p, \beta}.$
		\end{proof}
		Before we close this section, we recall two results from \cite{Huang} that are crucial to the proof of the extrapolation theorem. These results are a consequence of the reverse H\"{o}lder inequality for the capacitary Muckenhoupt weights.
		\begin{theorem}[\cite{Huang}]
			\label{ImprovingPower}
			Let $p \geq 1$ and $w \in A_{p, \beta}$. If $w$ is quasi-continuous, then there exists a positive constant $\gamma$ such that $w^{1 + \gamma} \in A_{p, \beta}$.
		\end{theorem}
		\begin{theorem}[\cite{Huang}]
			\label{SelfImproving}
			Let $1 < p < \infty$ and $w \in A_{p, \beta}$ be a quasi-continuous weight. Then, there is a $q$ with $1 < q < p$ such that $w \in A_{q, \beta}$.
		\end{theorem}
	\section{The Extrapolation Theorem}
		\label{MainResultSection}
		In this section, we prove an extrapolation theorem using the Muckenhoupt $A_{p, \beta}$ weights in the setting of Hausdorff capacities. The techniques we use here are analogous to those found in \cite{DuoandikoetxeaBook}. 
		We first give a method to construct $A_{1, \beta}$ weights using the maximal operator $M_{\beta}$. In order to do so, we require the following result.
		\begin{lemma}
			\label{PreparatoryLemmaConstructionA1}
			Let $T$ be an operator acting on functions on $\mathbb{R}^n$ such that there is a constant $K > 0$ such that for every $f \in L^1 \left( \mathbb{R}^n, H_{\infty}^{\beta} \right)$, we have,
			\begin{equation}
				\label{Weak11T}
				H_{\infty}^{\beta} \left( \left\lbrace x \in \mathbb{R}^n | \left| Tf \left( x \right) \right| > t \right\rbrace \right) \leq \frac{K}{t} \| f \|_{L^1 \left( \mathbb{R}^n, H^{\beta}_{\infty} \right)}.
			\end{equation}
			Then, For every $0 < \gamma < 1$, there is a constant $K' > 0$ depending only on $\gamma$ such that for any set $E \subseteq \mathbb{R}^n$ with $H^{\beta}_{\infty} \left( E \right) < + \infty$, we have,
			\begin{equation}
				\label{PreparatoryEquation}
				\int\limits_{E} \left| Tf \right|^{\gamma} \mathrm{d}H^{\beta}_{\infty} \leq K' \left( H^{\beta}_{\infty} \left( E \right) \right)^{1 - \gamma} \| f \|_{L^1 \left( \mathbb{R}^n, H^{\beta}_{\infty} \right)}^{\gamma}.
			\end{equation}
		\end{lemma}
		\begin{proof}
			Using the definition of Choquet integral (see Equation \eqref{Choquet}) and changing variables $t$ to $t^{\gamma}$, we easily see that
			$$\int\limits_{E} \left| Tf \right|^{\gamma} \mathrm{d}H^{\beta}_{\infty} = \gamma \int\limits_{0}^{\infty} t^{\gamma - 1} H^{\beta}_{\infty} \left( \left\lbrace x \in E | \left| Tf \left( x \right) \right| > t \right\rbrace \right) \mathrm{d}t.$$
			Due to Equation \eqref{Weak11T} and the Monotonicity of the Hausdorff content, we see that
			$$H^{\beta}_{\infty} \left( \left\lbrace x \in E | \left| Tf \left( x \right) \right| > t \right\rbrace \right) \leq \min \left\lbrace H^{\beta}_{\infty} \left( E \right), \frac{K}{t} \| f \|_{L^1 \left( \mathbb{R}^n, H^{\beta}_{\infty} \right)} \right\rbrace.$$
			Thus, we get
			\begin{align*}
				\int\limits_{E} \left| Tf \right|^{\gamma} \mathrm{d}H^{\beta}_{\infty} &\leq \gamma \left[ \int\limits_{0}^{\frac{K \| f \|_{L^1 \left( \mathbb{R}^n, H^{\beta}_{\infty} \right)}}{H^{\beta}_{\infty} \left( E \right)}} t^{\gamma - 1} H^{\beta}_{\infty} \left( E \right) \mathrm{d}t \right. \\
				&\left. + \int\limits_{\frac{K \| f \|_{L^1 \left( \mathbb{R}^n, H^{\beta}_{\infty} \right)}}{H^{\beta}_{\infty} \left( E \right)}}^{\infty} t^{\gamma - 2} K \| f \|_{L^1 \left( \mathbb{R}^n, H^{\beta}_{\infty} \right)} \mathrm{d}t \right].
			\end{align*}
			Upon simplifying, we obtain Equation \eqref{PreparatoryEquation} with $K' = \frac{K^{\gamma}}{1 - \gamma}$.
		\end{proof}
		\begin{theorem}
			\label{ConstructionA1Weights}
			Let $f \in L^1_{\text{loc}} \left( H^{\beta}_{\infty} \right)$ such that $M_{\beta}f \left( x \right) < + \infty$ for $H^{\beta}_{\infty}$-quasi-every $x \in \mathbb{R}^n$. Then, for every $\delta \in \left[ 0, 1 \right)$, the function $\left( M_{\beta}f \right)^{\delta}$ is an $A_{1, \beta}$ weight.
		\end{theorem}
		\begin{proof}
			Given $x \in \mathbb{R}^n$, we fix a cube $Q \subseteq \mathbb{R}^n$ with $x \in Q$. To show that $\left( M_{\beta}f \right)^{\delta}$ satisfies the $A_{1, \beta}$ condition, it is enough to show that with $Q$ fixed arbitrarily, we have,
			\begin{equation}
				\label{GoalEquation}
				\frac{1}{H^{\beta}_{\infty} \left( Q \right)} \int\limits_{Q} \left( M_{\beta}f \right)^{\delta} \ \mathrm{d}H^{\beta}_{\infty} \leq K \left( M_{\beta}f \right)^{\delta} \left( x \right),
			\end{equation}
			for some constant $K > 0$. In order to do so, we decompose $f$ as
			$$f = f \cdot \chi_{2Q} + f \cdot \chi_{\mathbb{R}^n \setminus \left( 2Q \right)},$$
			where, $2Q$ is the cube with the same center as that of $Q$ but twice its length. For the simplicity of notations, let us write $f_1 = f \cdot \chi_{2Q}$ and $f_2 = f \cdot \chi_{\mathbb{R}^n \setminus \left( 2Q \right)}$. Since the Choquet integral for the dyadic Hausdorff content $H^{\beta, Q_0}_{\infty}$ is sublinear and the dyadic Hausdorff content is equivalent to $H^{\beta}_{\infty}$, we have,
			$$M_{\beta}f \left( y \right) \leq K \left( M_{\beta}f_1 \left( y \right) + M_{\beta}f_2 \left( y \right) \right),$$
			for some constant $K > 0$ that depends only on the dimensions $n$ and $\beta$. Further, since $0 \leq \delta < 1$, we see that
			$$\left( M_{\beta}f \right)^{\delta} \left( y \right) \leq K^{\delta} \left( \left( M_{\beta}f_1 \right)^{\delta} \left( y \right) + \left( M_{\beta}f_2 \right)^{\delta} \left( y \right) \right),$$
			for every $y \in \mathbb{R}^n$. From this point onward, we use the alphabet $K$ to denote a positive constant that depends only on the parameters of the statement, and whose value might be different at every occurrence.
			
			Since $M_{\beta}$ is satisfies the weak type $(1, 1)$ inequality with respect to $H^{\beta}_{\infty}$, we have from Lemma \ref{PreparatoryLemmaConstructionA1},
			$$\frac{1}{H^{\beta}_{\infty} \left( Q \right)} \int\limits_{Q} \left( Mf_1 \right)^{\delta} \mathrm{d}H^{\beta}_{\infty} \leq K \left( \frac{1}{H^{\beta}_{\infty} \left( Q \right)} \int\limits_{2Q} \left| f_1 \right| \mathrm{d}H^{\beta}_{\infty} \right)^{\delta}.$$
			Now, we interpret (upto a constant) $H^{\beta}_{\infty}$ as the cubic Hausdorff content (see Definition \ref{CubicHC}) so that $H^{\beta}_{\infty} \left( Q \right) = \left( \ell \left( Q \right) \right)^{\beta} = \frac{H^{\beta}_{\infty} \left( 2Q \right)}{2^{\beta}}$. That is, we have,
			\begin{equation}
				\label{F1Equation}
				\frac{1}{H^{\beta}_{\infty} \left( Q \right)} \int\limits_{Q} \left( Mf_1 \right)^{\delta} \mathrm{d}H^{\beta}_{\infty} \leq K \left( \frac{1}{H^{\beta}_{\infty} \left( 2Q \right)} \int\limits_{2Q} \left| f \right| \mathrm{d}H^{\beta}_{\infty} \right)^{\delta} \leq K \left( M_{\beta}f \right)^{\delta} \left( x \right).
			\end{equation}
			Next, we take a cube $Q'$ containing a point $y \in Q$ such that $\int\limits_{Q'} \left| f_2 \right| \mathrm{d}H^{\beta}_{\infty} > 0$. Since $f_2$ is supported outside of $2Q$, we must have that the length of $Q'$ is at least half that of $Q$, for if not $Q' \subseteq 2Q$ and the integral would be zero. This also gives that $2Q \subseteq K \cdot Q'$, for a constant $K > 0$ that depends only on the dimension $n$ (for instance, one may take $K = 4 \sqrt{n}$). This gives that $K \cdot Q$ contains $x$.
			
			Thus, for any $y \in Q$ and a cube $Q'$ containing $y$ with $\int\limits_{Q'} \left| f_2 \right| \mathrm{d}H^{\beta}_{\infty} > 0$, we must have,
			$$\frac{1}{H^{\beta}_{\infty} \left( Q' \right)} \int\limits_{Q'} \left| f_2 \right| \mathrm{d}H^{\beta}_{\infty} \leq \frac{K}{H^{\beta}_{\infty} \left( K \cdot Q \right)} \int\limits_{K \cdot Q} \left| f \right| \mathrm{d}H^{\beta}_{\infty} \leq K M_{\beta}f \left( x \right).$$
			Therefore,
			$$M_{\beta}f_2 \left( y \right) \leq K \ M_{\beta}f \left( x \right),$$
			for every $y \in Q$. This leads us to,
			\begin{equation}
				\label{F2Equation}
				\frac{1}{H^{\beta}_{\infty} \left( Q \right)} \int\limits_{Q} \left( M_{\beta}f_2 \right)^{\delta} \left( y \right) \mathrm{d}H^{\beta}_{\infty} \left( y \right) \leq K \left( M_{\beta}f \right)^{\delta} \left( x \right).
			\end{equation}
			Combining Equations \eqref{F1Equation} and \eqref{F2Equation} and using the fact that the Hausdorff content and the dyadic Hausdorff content are equivalent with the Choquet integral of the latter being sublinear, we get Equation \eqref{GoalEquation}. This completes the proof!
		\end{proof}
		We are now in a position to give the extrapolation theorem. The following analogue of Jones factorization theorem in the capacitary setting (see Theorem 1.10 of \cite{Huang}) is important in our proof of the extrapolation theorem.
		\begin{theorem}[Jones Factorization Theorem]
			\label{JFT}
			For $p \geq 1$, a weight function $w \in A_{p, \beta}$ is and only if there are weights $w_0, w_1 \in A_{1, \beta}$ such that $w = w_0 w_1^{1 - p}$.
		\end{theorem}
		\begin{theorem}[Extrapolation Theorem]
			\label{ExtrapolationTheorem}
			Let $1 < p_0 < \infty$ be fixed. Let $T$ be an operator that is bounded on $L^{p_0} \left( \mathbb{R}^n, w \mathrm{d}H^{\beta}_{\infty} \right)$, for every $w \in A_{p, \beta}$ with the operator norm depending only on the $A_{p, \beta}$-constant of $w$. Then, we have the following.
			\begin{enumerate}
				\item[(A)] For any $1 < p < p_0$ and any $w \in A_{1, \beta}$, the operator $T$ is bounded on $L^p \left( \mathbb{R}^n, w \ \mathrm{d}H^{\beta}_{\infty} \right)$.
				\item[(B)] For any $1 < p < \infty$ and every quasi-continuous $w \in A_{p, \beta}$, the operator $T$ is bounded on $L^p \left( \mathbb{R}^n, w \ \mathrm{d}H^{\beta}_{\infty} \right)$.
			\end{enumerate}
		\end{theorem}
		\begin{proof}~
			\begin{enumerate}
				\item[(A)] Let us fix $1 < p < \infty$ and $w \in A_{1, \beta}$. Then, we have that $w \in A_{p, \beta}$ (see Remark 1.5 of \cite{Huang}). From Theorem \ref{ApWeightsTheorem}, we know that $M_{\beta}$ is bounded on $L^p \left( \mathbb{R}^n, w \ \mathrm{d}H^{\beta}_{\infty} \right)$. Hence, for any $f \in L^p \left( \mathbb{R}^n, w \ \mathrm{d}H^{\beta}_{\infty} \right)$, $M_{\beta} f \left( x \right)$ is finite $H^{\beta}_{\infty}$-quasi-everywhere. Therefore, from Lemma \ref{PreparatoryLemmaConstructionA1}, we have that $\left( M_{\beta}f \right)^{\frac{p_0 - p}{p_0 - 1}} \in A_{1, \beta}$. By the Jones factorization theorem, we have that $w \left( M_{\beta}f \right)^{p - p_0} \in A_{p, \beta}$. Now, using the equivalence of $H^{\beta}_{\infty}$ and $H^{\beta, Q_0}_{\infty}$ and H\"{o}lder's inequality with exponents $\left( \frac{p_0}{p} \right)$ and $\left( \frac{p_0}{p} \right)'$, we get
				\begin{align*}
					&\int\limits_{\mathbb{R}^n} \left| Tf \right|^p w \ \mathrm{d}H^{\beta}_{\infty} \\
					&\leq K \left( \int\limits_{\mathbb{R}^n} \left| Tf \right|^{p_0} \left( M_{\beta}f \right)^{p - p_0} w \ \mathrm{d}H^{\beta}_{\infty} \right)^{\frac{p}{p_0}} \left( \int\limits_{\mathbb{R}^n} \left( M_{\beta}f \right)^p w \ \mathrm{d}H^{\beta}_{\infty} \right)^{1 - \frac{p}{p_0}}.
				\end{align*}
				Using the boundedness of $T$ on $L^{p_0} \left( \mathbb{R}^n, w \left( M_{\beta}f \right)^{p - p_0} \mathrm{d}H^{\beta}_{\infty} \right)$ and that of $M_{\beta}$ on $L^p \left( \mathbb{R}^n, w \ \mathrm{d}H^{\beta}_{\infty} \right)$, we get
				$$\int\limits_{\mathbb{R}^n} \left| Tf \right|^p w \ \mathrm{d}H^{\beta}_{\infty} \leq K \left( \int\limits_{\mathbb{R}^n} \left| f \right|^{p_0} \left( M_{\beta}f \right)^{p - p_0} w \ \mathrm{d}H^{\beta}_{\infty} \right)^{\frac{p}{p_0}} \left( \int\limits_{\mathbb{R}^n} \left| f \right|^p w \ \mathrm{d}H^{\beta}_{\infty} \right)^{1 - \frac{p}{p_0}}.$$
				Finally, from Corollary \ref{LDTCor} and the fact that $p < p_0$, we get
				$$\int\limits_{\mathbb{R}^n} \left| Tf \right|^p w \ \mathrm{d}H^{\beta}_{\infty} \leq K \left( \int\limits_{\mathbb{R}^n} \left| f \right|^p w \ \mathrm{d}H^{\beta}_{\infty} \right).$$
				\item[(B)] To prove this part of the result, we use Lemma \ref{DualityLemma}. First, let us observe that for any $1 < p_1 < p$, we have,
				$$\| Tf \|_{L^p \left( \mathbb{R}^n, w \ \mathrm{d}H^{\beta}_{\infty} \right)}^{p_1} = \| |Tf|^{p_1} \|_{L^{\frac{p}{p_1}} \left( \mathbb{R}^n, w \ \mathrm{d}H^{\beta}_{\infty} \right)}.$$
				Due to Lemma \ref{DualityLemma}, it is enough to show for every $u \in L^{\left( \frac{p}{p_1} \right)'} \left( \mathbb{R}^n, w \ \mathrm{d}H^{\beta}_{\infty} \right)$ with $u \geq 0$ and $\| u \|_{L^{\left( \frac{p}{p_1} \right)'} \left( \mathbb{R}^n, w \ \mathrm{d}H^{\beta}_{\infty} \right)} = 1$, that,
				$$\int\limits_{\mathbb{R}^n} \left| Tf \right|^{p_1} u w \ \mathrm{d}H^{\beta}_{\infty} \leq K \| f \|_{L^p \left( \mathbb{R}^n, w \ \mathrm{d}H^{\beta}_{\infty} \right)}^{p_1}.$$
				We now begin the proof of (B). For the same, we fix $1 < p_1 < \min \left\lbrace p, p_0 \right\rbrace$ and a quasi-continuous $w \in A_{\frac{p}{p_1}, \beta}$. We also fix a $u \in L^{\left( \frac{p}{p_1} \right)'} \left( \mathbb{R}^n, w \ \mathrm{d}H^{\beta}_{\infty} \right)$ as above. Then, from Theorem \ref{ImprovingPower}, we get some $s_0 > 1$ such that 
				$$w^{1 + \left( \frac{p}{p_1} \right)' \frac{\left( s_0 - 1 \right)}{\left( \frac{p}{p_1} \right)' - s_0}} \in A_{\frac{p}{p_1}, \beta}.$$
				Similarly, due to Theorems \ref{ApApPrimeCharacterization} and \ref{SelfImproving}, we get some $s_1 > 1$ such that $w^{1 - \left( \frac{p}{p_1} \right)'} \in A_{\frac{\left( \frac{p}{p_1} \right)'}{s_1}, \beta}$. We may assume that $s_0 = s_1 = s$, by taking a maximum. That is, we have some $s > 1$ such that $w^{1 + \left( \frac{p}{p_1} \right)' \frac{\left( s - 1 \right)}{\left( \frac{p}{p_1} \right)' - s}} \in A_{\frac{p}{p_1}, \beta}$ and $w^{1 - \left( \frac{p}{p_1} \right)'} \in A_{\frac{\left( \frac{p}{p_1} \right)'}{s}, \beta}$.
				
				Now, using H\"{o}lder's inequality with exponents $r = \frac{\left( \frac{p}{p_1} \right)'}{s}$ and $r'$, we get for every compact set $K \subseteq \mathbb{R}^n$,
				\begin{align*}
					\int\limits_{K} u^sw^s \mathrm{d}H^{\beta, Q_0}_{\infty} &\leq \left( \int\limits_{K} u^{\left( \frac{p}{p_1} \right)'} w \ \mathrm{d}H^{\beta, Q_0}_{\infty} \right)^{\frac{1}{r}} \left( \int\limits_{K} w^{1 + r' \left( s - 1 \right)} \mathrm{d}H^{\beta, Q_0}_{\infty} \right)^{\frac{1}{r'}} \\
					&= \left( \int\limits_{K} w^{1 + r' \left( s - 1 \right)} \mathrm{d}H^{\beta, Q_0}_{\infty} \right)^{\frac{1}{r'}}.
				\end{align*}
				By the choice of $s$, we know that $w^{1 + r' \left( s - 1 \right)}$ is locally integrable and therefore, $u^sw^s$ is locally integrable. Using Lemma \ref{ConstructionA1Weights}, we have that $\left( M_{\beta} \left( u^sw^s \right) \right)^{\frac{1}{s}} \in A_{1, \beta}$. Further, due to Corollary \ref{LDTCor} and the (reverse) Jensen's inequality (see Theorem 2.1 of \cite{ZhangJensen}), we get
				$$uw \leq M_{\beta}\left( uw \right) \leq \left( M_{\beta}\left( u^sw^s \right) \right)^{\frac{1}{s}}.$$
				Using this observation and part (A) of the theorem, we have,
				\begin{align*}
					\int\limits_{\mathbb{R}^n} \left| Tf \right|^{p_1} uw \mathrm{d}H^{\beta, Q_0}_{\infty} &\leq K \int\limits_{\mathbb{R}^n} \left| Tf \right|^{p_1} \left( M_{\beta} \left( u^sw^s \right) \right)^{\frac{1}{s}} \mathrm{d}H^{\beta, Q_0}_{\infty} \\
					&\leq K \int\limits_{\mathbb{R}^n} \left| f \right|^{p_1} \left( M_{\beta}\left( u^sw^s \right) \right)^{\frac{1}{s}} \mathrm{d}H^{\beta, Q_0}_{\infty}.
				\end{align*}
				Finally, using the H\"{o}lder inequality with exponents $\frac{p}{p_1}$ and $\left( \frac{p}{p_1} \right)'$, we get
				\begin{align*}
					&\int\limits_{\mathbb{R}^n} \left| Tf \right|^{p_1} uw \mathrm{d}H^{\beta, Q_0}_{\infty} \\
					&\leq K \left( \int\limits_{\mathbb{R}^n} \left| f \right|^{p} w \ \mathrm{d}H^{\beta, Q_0}_{\infty} \right)^{\frac{p_1}{p}} \left( \int\limits_{\mathbb{R}^n} \left( M_{\beta} \left( u^sw^s \right) \right)^{\frac{\left( \frac{p}{p_1} \right)'}{s}} w^{1 - \left( \frac{p}{p_1} \right)'} \mathrm{d}H^{\beta, Q_0}_{\infty} \right)^{\frac{1}{\left( \frac{p}{p_1} \right)'}}.
				\end{align*}
				Now, we recall that $w^{1 - \left( \frac{p}{p_1} \right)'} \in A_{\frac{\left( \frac{p}{p_1} \right)'}{s}, \beta}$ so that by the equivalence of $H^{\beta, Q_0}_{\infty}$ and $H^{\beta}_{\infty}$ and Theorem \ref{ApWeightsTheorem}, we get
				\begin{align*}
					&\left( \int\limits_{\mathbb{R}^n} \left( M_{\beta} \left( u^sw^s \right) \right)^{\frac{\left( \frac{p}{p_1} \right)'}{s}} w^{1 - \left( \frac{p}{p_1} \right)'} \mathrm{d}H^{\beta, Q_0}_{\infty} \right)^{\frac{1}{\left( \frac{p}{p_1} \right)'}} \\
					&\leq K \left( \int\limits_{\mathbb{R}^n} u^{\left( \frac{p}{p_1} \right)'} w \ \mathrm{d}H^{\beta, Q_0}_{\infty} \right)^{\frac{1}{\left( \frac{p}{p_1} \right)'}} = K,
				\end{align*}
				due to the choice of $u$. Thus, we have shown that (by taking a supremum over $u$) for any $1 < p_1 < \min \left\lbrace p, p_0 \right\rbrace$ and a quasi-continuous weight $w \in A_{\frac{p}{p_1}, \beta}$, the operator $T$ is bounded on $L^p \left( \mathbb{R}^n, w \ \mathrm{d}H^{\beta}_{\infty} \right)$.
				
				Now, let $w \in A_{p, \beta}$ be a quasi-continuous weight. Then, by Theorem \ref{SelfImproving}, there is some $1 < p_1 < p$ such that $w \in A_{\frac{p}{p_1}, \beta}$. Due to the monotonicity of the classes $A_{p, \beta}$, we may assume that $p_1 < p_0$ so that $1 < p_1 < \min \left\lbrace p, p_0 \right\rbrace$. This leads us to the previous analysis and completes the proof of the extrapolation result!
			\end{enumerate}
		\end{proof}
	\section*{Acknowledgement}
		The author extends his gratitude towards the Ministry of Education, Government of India, for awarding the Prime Minister Research Fellowship (PMRF), grant number: 2101706, that has funded this research.

\bibliographystyle{amsplain}
		
\end{document}